\documentclass{article}
\usepackage[utf8]{inputenc}

\usepackage{graphicx,graphics}

\usepackage{rotating}
\usepackage[twoside]{geometry}
\usepackage{epsfig}
\usepackage{indentfirst}
\usepackage[usenames,dvipsnames,svgnames,table]{xcolor}
\usepackage{cmap}

\usepackage{graphicx,graphics}

\usepackage{hyperref}
\hypersetup{
pdftitle={Semigroups of operators},    
pdfauthor={Moacir Aloisio},    
pdfnewwindow=true,      
colorlinks=true,      
linkcolor=black,         
citecolor=blue}

\usepackage{amssymb,amsfonts,amstext,amsthm}
\usepackage{mathrsfs}
\usepackage[intlimits]{amsmath}
\newtheorem{definition}{\bf Definition}[section]
\newtheorem{proposition}{\bf Proposition}[section]

\newtheorem{lemma}{\bf Lemma}[section]
\newtheorem{theorem}{\bf Theorem}[section]

\newtheorem{corollary}{\bf Corollary}[section]

\newtheorem{remark}{Remark}[section]

\newtheorem{example}{Example}[section]

\newcounter{H1}
\newenvironment{H1}{\refstepcounter{H1}\equation}{\tag{H}\endequation}

\newcounter{ACP1}
\newenvironment{ACP1}{\refstepcounter{ACP1}\equation}{\tag{ACP}\endequation}

\title{Ref\/ined scales of decaying rates of operator semigroups on Hilbert spaces: typical behaviour}
\author{M. Aloisio, S. L. Carvalho, and C. R. de Oliveira\thanks{Corresponding author. Telephone +55 16 3351 9153, fax +55 16 3361 2081.}}
\date{October 2019}

\begin{document}

\maketitle

\begin{abstract} We study relations between the decaying rates of operator semigroups on Hilbert spaces and some spectral properties of their respective generators; in particular, we show that the decaying rates of orbits of semigroups which are stable but not exponentially stable, typically in Baire's sense, depend on sequences of time going to infinity.
\end{abstract} 

%%%%%%%%%%%%%%%%%%%%%%%%%%%%%%%%%%%%%%%%%%%%%%%%%%%%%%%%%%%%%%%%%%%%%%%%%%%--Introduction--%%%%%%%%%%%%%%%%%%%%%%%%%%%%%%%%%%%%%%%%%%%%%%%%%%%%%%%%%%%%%%%%%%%%%%%%%%%%%%%%%%%%%%%%%%%%%%%%%%%%%%%%%%%%%%%%%%%%%%%%%%%%%%%%%%%%%%%%%%%%%%%%%%%%%%%%%%%%%%%%%%%%%%%%%%%%%%%%%%%%%%%%%%%%%%%%%%%%%%%%%%%%%%%%%%%%%%%%%%%%%%%%%%%%%%%%%%%%%%%%%%%%%%%%%%%%%%%%%%

\renewcommand{\thetable}{\Alph{table}}

\section{Introduction}

%%%%%%%%%%%%%%%%%%%%%%%%%%%%%%%%%%%%%%%%%%%%%%%%%%%%%%%%%%%%%%%%%%%%%%%%%--Contextualization--%%%%%%%%%%%%%%%%%%%%%%%%%%%%%%%%%%%%%%%%%%%%%%%%%%%%%%%%%%%%%%%%%%%%%%%%%%%%%%%

\noindent There is a vast literature concerning the theory of stability for solutions of the abstract Cauchy problem
\begin{ACP1}\label{cauchyproblem}
\begin{cases} \dot{x}(t) = A x(t), \quad t \geq 0, \\ x(0) = x, \quad x \in {\mathcal{H}},  \end{cases}
\end{ACP1}
where~$A$ is the generator of a $C_0$-semigroup~$(T(t))_{t\geq 0}$ on a Hilbert space ${\mathcal{H}}$; such theory has a very rich intersection with spectral theory, harmonic analysis, and numerous applications to partial differential equations. Actually, a great part of the recent interest in the study of stability of solutions of (\ref{cauchyproblem}) comes from damped wave equations and other models (see \cite{Batty,Borichev,Rozendaal,van} and references therein). 

Among the important results in this field we highlight the characterization of exponential stability for $C_0$-semigroups on Hilbert spaces (Gearhart-Pr\"{u}ss Theorem) due to Herbst, Howland and Pr\"{u}ss \cite{Herbst,Howland,Pruss}, the stability theorem, by Arendt, Batty, Lyubich and V\~u \cite{Arendt,Lyubich}, stating that a bounded $C_0$-semigroup on a Banach space is (strongly) stable if the peripheral spectrum of its generator is countable and contains no residual spectrum. We also highlight the recent results obtained by Borichev and Tomilov \cite{Borichev}, by Batty, Chill and Tomilov \cite{Batty} and very recently by Rozendaal, Seifert, and Stahn \cite{Rozendaal}, relating estimates on the norm of the resolvent of the generator to quantitative rates of convergence of the form
\begin{equation*}
\|T(t)A^{-1}\|_{\mathcal{B}(\mathcal{H})} = O(r(t)), \quad t \rightarrow \infty,
\end{equation*}
with $\displaystyle\lim_{t\to\infty}r(t)=0$, developed in order to explore polynomial and logarithmic scales, among others, of decaying rates of bounded $C_0$-semigroups. As it is well known, this strategy has allowed numerous applications of the theory to PDEs; namely, estimates on the norm of the resolvent of the generator are often easier to compute than the estimates on the norm of the semigroup itself. In this context, we refer to \cite{Anantharaman,Batty,Burq,Conti,Rozendaal}, among others.

An important intermediate step from Gearhart-Pr\"{u}ss Theorem to results by Rozendaal et al.~\cite{Rozendaal} was the Batty-Duyckaerts Theorem (Theorem \ref{Duyckaertstheorem} below) which relates the decaying rates of $\|T(t)A^{-1}\|_{\mathcal{B}(X)}$, $i {\mathbb{R}}  \subset \varrho(A)$, with the arbitrary growth of the norm of the resolvent of the generator. In order to properly recall such result, we need some preliminaries.

For every $A$, the generator of a bounded $C_0$-semigroup $(T(t))_{t \geq 0}$ on a Banach space $X$, with $i {\mathbb{R}}  \subset \varrho(A)$, consider a continuous non-decreasing function
\[M(y) := \max_{\lambda \in [-y,y]} \|R(i\lambda,A)\|_{\mathcal{B}(X)}, \quad y \geq 0,\]
and the associated function
\[ M_{\log}(y) := M(y) (\log(1 + M(y)) + \log(1 + y)), \quad  y \geq 0.  \] 
Denote by $M_{\log}^{-1}:[M_{\log}(0), \infty)\rightarrow\mathbb{R}$ the inverse of $M_{\log}$.

\begin{theorem}[Theorem 1.5 in \cite{Batty1}]\label{Duyckaertstheorem} Let ~$(T(t))_{t\geq 0}$ be a bounded ~$C_0$-semigroup on a Banach space~$X$, with generator~$A$ such that $i {\mathbb{R}}  \subset \varrho(A)$. Then, there exists $C>0$ such that
\begin{equation}\label{Duyckaerts}
\|T(t)A^{-1}\|_{\mathcal{B}(X)} = O(1/M_{\log}^{-1}(t/C)), \quad t \rightarrow \infty.   
\end{equation}
\end{theorem}

For a refinement of (\ref{Duyckaerts}) on Hilbert spaces, see \cite{Batty,Borichev} (see also Theorem \ref{theopolstab} ahead).
 
The proof of Theorem \ref{Duyckaertstheorem} by Batty and Duyckaerts \cite{Batty1}, which uses a technique developed by Korevaar \cite{Korevaar}, makes use of Cauchy's Theorem and Neumann series expansions. Usually, the problem of obtaining lower bounds for the decaying rates of stable bounded $C_0$-semigroups passes through the understanding of some theory of integral representation (like, for instance, Cauchy's theory and  the functional calculus of sectorial operators \cite{Batty,Batty1}). In this paper, we use the joint resolution of the identity for normal operators \cite{Berezannskii,Samoilenko} and a result due to Muller and Tomilov \cite{Muller} (see Theorem \ref{theoMuller1} ahead) to find the typical asymptotic behavior, in Baire's sense, of the orbits of normal $C_0$-semigroups of contractions. We also  say something about non-normal semigroups. To the best of our knowledge, none of this has been detailed in the literature yet. 

%%%%%%%%%%%%%%%%%%%%%%%%%%%%%%%%%%%%%%%%%%%%%%%%%%%%%%%%%%%%%%--Brief--discussion--of--our--main--results--%%%%%%%%%%%%%%%%%%%%%%%%%%%%%%%%%%%%%%%%%%%%%%%%%%%%%%%%%%%%%%%%%%

\subsection{Brief discussion of our main results}

\noindent Let, for every $\lambda \in\mathbb {\mathbb{C}}$ and every $t \geq 0,$ $g_{t}(\lambda) = e^{t\lambda}$, and let~$N$ be a normal operator on a Hilbert space~$\mathcal{H}$; denote by $\varrho(N)$ and~$\sigma(N)$ the resolvent set and spectrum of~$N$, respectively, and by $ R(\lambda)$ the real part of the complex number~$\lambda$. We note that if $ \{ z \in {\mathbb{C}} : R(z) > 0\} =: {\mathbb{C}}_+ \subset \varrho(N)$, then, by the spectral functional calculus, $(e^{tN})_{t \geq 0}:=(g_t(N))_{t\geq 0}$ is a normal $C_0$-semigroup of contractions generated by~$N$. It is well known that every normal $C_0$-semigroup of contractions is of this form~\cite{Rudin}.

As discussed previously, after Batty and Duyckaerts \cite{Batty1} have related the decaying rates of stable bounded $C_0$-semigroups to the arbitrary growth of the norms of the respective resolvents, the study of polynomial and logarithmic scales of such rates has been the subject of many recent papers (see \cite{Batty,Borichev,Rozendaal} and references therein). In contrast with this setting, our first result says that the decaying rates of the orbits of normal $C_0$-semigroups of contractions, typically in Baire's sense, may depend on sequences of time going to infinity. 

\begin{theorem}\label{maintheoremnormal} Let~$N$ be a normal operator in $\mathcal{H}$ such that $\sup \{ R(\lambda): \lambda \in \sigma(N)\} = 0$ and let $\alpha,\beta:{\mathbb{R}}_+ \longrightarrow (0,\infty)$ be real functions so that
\[\lim_{t \to \infty} \alpha(t) = \infty \quad and \quad \lim_{t \to \infty} \beta(t)e^{-t \epsilon }= 0, \ \forall \epsilon>0.\]
 Suppose that $(e^{tN})_{t\geq 0}$ is stable. Then, 
\[{\mathcal{G}}_N(\alpha,\beta) :=\{x \displaystyle\mid\limsup_{t \to \infty} \alpha(t) \Vert e^{tN}x\Vert_{\mathcal{H}} = \infty \quad and \quad \liminf_{t \to \infty} \beta(t) \Vert e^{tN}x\Vert_{\mathcal{H}} = 0\}\]
is a dense $G_\delta$ set in~$\mathcal{H}$. Moreover, the assumption on $\beta$ is optimal, that is, $\beta$ cannot be chosen to grow faster than sub-exponentially. 
\end{theorem}

The next example shows that the asymptotic behavior of the orbit of a classical solution to (\ref{cauchyproblem}), in Baire's sense, may be atypical.

\begin{example}\label{example1} {\rm Let $\varphi: [2,\infty) \longrightarrow \mathbb{C}$ given by the action $\varphi(y) := \frac{1}{\ln y} + iy$, and then define $M_\varphi : {\mathcal D}(M_\varphi)\subset {\mathrm L}^2[2,\infty) \longrightarrow {\mathrm L}^2[2,\infty)$, 
\[(M_\varphi f)(y) = -\varphi(y) f(y),\] 
where $f\in{\mathcal{D}}(M_\varphi) := \{u \in {\mathrm L}^2[2,\infty) \mid  \varphi u \in {\mathrm L}^2[2,\infty)\}.$ $M_\varphi$ is a normal operator such that $\sigma(M_\varphi) = \{ -\frac{1}{\ln y} - iy, ~ y \geq 2\}$, which implies that $\sup \{ R(\lambda): \lambda \in \sigma(M_\varphi)\} = 0$. 

It is possible to show that  
\begin{equation*}
\Vert e^{tM_\varphi}M_\varphi^{-1}\Vert_{{\mathcal{B}}({\mathrm L}^2[2,\infty))} = O(e^{-2\sqrt{t}})    
\end{equation*}
(this is Example 5.2 in~\cite{Batty}). In particular, there is $C>0$ such that, for each $x \in {\mathcal{D}}(M_\varphi)$ and each $t \geq 0$, 
\begin{equation}\label{ex1}
\|e^{tM_\varphi}x\|_{{\mathcal{H}}} \leq  C e^{-2\sqrt{t}} \|x\|_{{\mathcal{H}}}.    
\end{equation}
We note that (\ref{ex1}) is not the typical behavior of an orbit of such semigroup. Namely, let $\alpha(t) = t$ and $\beta(t) = t$, for all $t > 0$. If (\ref{ex1}) holds, then $x \in \mathcal{G}_{M_\varphi}^c(\alpha,\beta)=\mathcal{H}\setminus\mathcal{G}_{M_\varphi}(\alpha,\beta)$, which is,  by Theorem~\ref{maintheoremnormal}, a countable union of closed sets of empty interior in $\mathcal{H}$.  

Finally, we also note that, by Theorem~\ref{maintheoremnormal}, each typical orbit goes to zero at a fixed but arbitrarily slow rate for a sequence of time going to infinity and sub-exponentially fast for another one.}
\end{example}

With respect to Theorem \ref{maintheoremnormal}, Example~\ref{example12}  illustrates  that in some cases very regular initial data may belong to the typical set ${\mathcal{G}}_N(\alpha,\beta)$ with erratic behavior.

\begin{example}\label{example12}{\rm Let $A: \ell^2({\mathbb{Z}}) \longrightarrow \ell^2({\mathbb{Z}})$ be the linear operator given by  
\[(Au)_n  = u_{n-1}+u_{n+1}, \ n \in \mathbb{Z}.\]
It is known that $A$ is unitarily equivalent to the multiplication operator ${\mathcal{M}}_{\phi}$ on ${\mathrm{L}}^2[0,2\pi)$, with $\phi(x) = 2\cos(x)$, from which  follows that $A$ is a bounded self-adjoint operator with continuous spectrum $\sigma(A) = \sigma({\mathcal{M}}_\phi) = \sigma_{c}({\mathcal{M}}_\phi) =  [-2,2]$ (see \cite{Oliveira} for details).

Now we consider the discrete Laplacian, $\triangle$, given on ${\ell}^2({{\mathbb{Z}}})$ by the action
\[(\triangle u)_n = (Au)_n -2 u_n.\]
So, $\triangle$ is a bounded self-adjoint operator with continuous spectrum $\sigma(\triangle) = [-4,0]$, from which it follows that $(e^{t\triangle})_{t \geq 0}$ is stable but not exponentially stable; therefore, $\triangle$ satisfies the hypotheses of Theorem \ref{maintheoremnormal}. We note that for every $x \in {\mathcal{G}}_\triangle(\alpha,\beta)$, $(e^{t\triangle}x)_{t \geq 0}$ is infinitely differentiable since the discrete Laplacian is a bounded linear operator.}
\end{example}

Now we recall that every normal operator $N$ can be written as $N = N_R + iN_I$, where $N_R$ and $N_I$ are self-adjoint operators such that $N_RN_I = N_IN_R$. The next theorem, a new spectral classification of (strong) stability for normal $C_0$-semigroups of contractions, is a direct application of Gearhart-Pr\"{u}ss Theorem and Theorem \ref{maintheoremnormal}.

\begin{theorem}\label{theoremnormalclas} Let~$N= N_R + iN_I$ be a normal operator in $\mathcal{H}$ so that ${\mathbb{C}}_+ \subset \varrho(N)$. Then:
\begin{enumerate}
\item[i)] All orbits of  $(e^{tN})_{t\geq 0}$ converge to zero with exponential rate if, and only if, $0 \not \in \sigma(N_R)$.
\item[ii)]The set ${\mathcal{G}}_N(\alpha,\beta)$ is a dense $G_\delta$ set in~$\mathcal{H}$, for every $\alpha$ and every $\beta$  as in the statement of Theorem \ref{maintheoremnormal}, if and only if $0 \in \sigma(N_R)$ and 0 is not an eigenvalue of $N_R$.
\item[iii)] The set $\mathcal{F}_N$ of $x \in \mathcal{H}$ such that $(e^{tN}x)_{t\geq 0}$ does not converge to zero is a dense open set in~$\mathcal{H}$ if, and only if, 0 is an eigenvalue of $N_R$.
\end{enumerate}  
\end{theorem}

Our next result is a partial extension of Theorem \ref{maintheoremnormal} to non-normal semigroups. We recall that the exponential growth bound of a $C_0$-semigroup $(T(t))_{t\geq 0}$ on~$\mathcal{H}$ is defined as \cite{van}
\[\omega_0(T) := \lim_{t \to \infty} \frac{\ln \|T(t)\|_{\mathcal{B}(\mathcal{H})}}{t}.\]

\begin{theorem}\label{maintheoremrefineddecay}Let~$(T(t))_{t\geq 0}$ be a bounded $C_0$-semigroup on~$\mathcal{H}$  such that $\omega_0(T) = 0$, and let $A$ be its generator. Suppose that $A$ is injective and, for some $k \geq 1$, that
\begin{H1}\label{H2}
\|T(t)A^{-k}\|_{\mathcal{B}(\mathcal{H})} = O(r(t)), \quad t \rightarrow \infty,
\end{H1}
with $\displaystyle\lim_{t\to\infty}r(t)=0$. Let $\alpha,\beta:{\mathbb{R}}_+ \longrightarrow (0,\infty)$ be real functions so that
\[\lim_{t \to \infty} \alpha(t) = \infty \quad and \quad \liminf_{t \to \infty} \beta(t)r(t)= 0.\]
Then, 
\begin{equation*}
{\mathcal{G}}_A(\alpha,\beta) =\{x \displaystyle\mid \limsup_{t \to \infty} \alpha(t) \Vert T(t)x\Vert_{\mathcal{H}} = \infty \ {\rm and} \ \liminf_{t \to \infty} \beta(t) \Vert T(t)x\Vert_{\mathcal{H}} = 0\}    
\end{equation*}
is a dense $G_\delta$ set in~$\mathcal{H}$ .
\end{theorem}

\begin{remark}\label{remarkmaintheoremrefineddecay}
\end{remark}
\begin{enumerate}

\item[i)] We note that, by Theorem \ref{Duyckaertstheorem}, if $i {\mathbb{R}}  \subset \varrho(A)$, then hypothesis~(\ref{H2}) is satisfied. Namely, it is possible to show that (\ref{H2}) holds if and only if $i {\mathbb{R}}  \subset \varrho(A)$; for details, see \cite{Batty1}.
    
\item[ii)] The main dynamical difference between Theorems \ref{maintheoremnormal} and \ref{maintheoremrefineddecay} is that the former, under the perspective of this work, describes the exact (large time) asymptotic behavior of normal $C_0$-semigroups of contractions, while the main spectral difference between Theorems \ref{maintheoremnormal} and \ref{maintheoremrefineddecay} is that in Theorem \ref{maintheoremnormal}, one may have $\sigma(N)\cap i\mathbb{R}\neq\emptyset$. Namely, thanks to the Spectral Theorem, we do not need to use hypothesis~(\ref{H2}) to prove it.     
    
\item[iii)]Suppose that there exists an $a>0$ such that $\|T(t)A^{-1}\|_{\mathcal{B}(\mathcal{H})} = O(t^{-a})$. Since, for every $k \geq 1$,
\[\|T(t)A^{-k}\|_{\mathcal{B}(\mathcal{H})} = \|[T(t/k)A^{-1}]^k\|_{\mathcal{B}(\mathcal{H})},\]
one has 
\[\|T(t)A^{-k}\|_{\mathcal{B}(\mathcal{H})} = O(t^{-ka}).\]
Thus, it follows from Theorem \ref{maintheoremrefineddecay} that
\[\bigcap_{k \geq 1}\{x  \mid \limsup_{t \to \infty} \alpha(t) \Vert T(t)x\Vert_{\mathcal{H}} = \infty \ {\rm and} \ \liminf_{t \to \infty} t^{ka/2}  \Vert T(t)x\Vert_{\mathcal{H}} = 0\}\]
is a dense $G_\delta$ set in ${\mathcal{H}}$. One should compare this with Example \ref{example1}: the asymptotic behavior of the orbit of a classical solution to (ACP), in Baire's sense, is atypical.
\end{enumerate}

Consider the following result. 

\begin{theorem}[Theorem 1.3 in~\cite{Batty}] \label{theopolstab} Let $(T(t))_{t\geq 0}$ be a bounded $C_0$-semigroup on a Hilbert space~$\mathcal{H}$, with generator~$A$, so that $i {\mathbb{R}}  \subset \varrho(A)$. Then, given $a>0$ and $b\in {\mathbb{R}}$, the following assertions are equivalent:  
\begin{enumerate}
\item[i)] $\|(is I - A)^{-1}\|_{\mathcal{B}(\mathcal{H})} = O(|s|^a(\ln|s|)^{-b}),  \quad |s| \rightarrow \infty,$
\item[ii)] $\|T(t)A^{-1}\|_{\mathcal{B}(\mathcal{H})} = O(t^{-\frac{1}{a}} (\ln t)^{-b/a}), \quad t \rightarrow \infty.$
\end{enumerate}
\end{theorem}

\begin{remark}
\end{remark}
\begin{enumerate}
\item[i)] Although the condition $b\ge 0$ is presented in the statement of Theorem \ref{theopolstab} in \cite{Batty}, it is not necessary to impose any restriction on $b$, as it was shown in Theorem 1.1 in \cite{Rozendaal}.
    
\item[ii)] We note that Theorem~\ref{maintheoremrefineddecay} can be naturally combined with Theorem~\ref{theopolstab} in order to produce refined scales of decaying rates of bounded $C_0$-semigroups. Namely, if we replace condition~\eqref{H2} in Theorem~\ref{maintheoremrefineddecay} by the condition depicted in item~ii) of Theorem~\ref{theopolstab}, then each typical orbit of the semigroup contains a sequence that decays to zero no faster than a fixed but arbitrarily slow rate, and, by Remark \ref{remarkmaintheoremrefineddecay}, a sequence that decays to zero at a rate super-polynomial. In this sense, we show that typical orbits display unexpected and erratic behavior.
     
\item[iii)] There are in the literature numerous examples of evolution equations whose associated semigroups satisfy the assumptions of Theorem~\ref{maintheoremrefineddecay}  (see \cite{Batty,Burq,Conti,Rozendaal} and references therein); particularly, such result also applies to some damped wave equations \cite{Anantharaman,Stahn}.
\end{enumerate}

%%%%%%%%%%%%%%%%%%%%%%%%%%%%%%%%%%%%%%%%%%%%%%%%%%%%%%%%%%%%%%%%%%%%%%%%%--Organization--of--text--%%%%%%%%%%%%%%%%%%%%%%%%%%%%%%%%%%%%%%%%%%%%%%%%%%%%%%%%%%%%%%%%%%%%%%%%%%%%

\subsection{Organization of text}

\noindent The organization of this paper is as follows. In Section \ref{Fine scales} we present a detailed study of the relation between the (polynomial) decaying rates of a normal semigroup and the local scale spectral properties of its generator; in particular, we prove Theorems~\ref{maintheoremnormal} and~\ref{maintheoremrefineddecay}. 

Some words about notation: $R(z)$ denotes the real part of $z \in \mathbb{C}$,  ${\mathbb{C}}_+ = \{ z \in {\mathbb{C}} : R(z) > 0\}$ and~$\mathcal{H}$ denotes a complex Hilbert space. Let $\mathcal{B}(\mathcal{H})$ denote the space of all bounded linear operators on~$\mathcal{H}$. If~$A$ is a linear operator in~$\mathcal{H}$, we denote its domain by ${\mathcal{D}}(A)$, its kernel by ${\rm kernel}(A)$, its spectrum by $\sigma(A)$  and the respective resolvent set by~$\varrho(A)$. For each $x \in {\mathbb{R}}$ and each $r>0$, $B(x,r)$ denotes the open interval $(x-r,x+r)$. Finally, $\| \cdot \|_{\mathcal{B}(\mathcal{H})}$ denotes the norm in $\mathcal{B}(\mathcal{H})$, and $\| \cdot \|_{\mathcal{H}}$  the norm in~$\mathcal{H}$.

%%%%%%%%%%%%%%%%%%%%%%%%%%%%%%%%%%%%%%%%%%%%%%%%%%%%%%%%%%%%%%%%%%%%%%%%%%%%%%%%%%%%%%%%%%%%%%%%%%%%%%%%%%%%%%%%%%%%%%%%%%%%%%%%%%%%%%%%%%%%%%%%%%%%%%%%%%%%%%%%%%%%%%%%%%%%%%%%%%%%%%%%%%%%%%%%%%%%%%%%%%%%%%%%%%%%%%%%%%%%%%%%%%%%%%%%%%%%%%%%--Fine--scales--of--decaying--rates--%%%%%%%%%%%%%%%%%%%%%%%%%%%%%%%%%%%%%%%%%%%%%%%%%%%%%%%%%%%%%%%%%%%%%%

\section{Fine scales of decaying rates}\label{Fine scales}

%%%%%%%%%%%%%%%%%%%%%%%%%%%%%%%%%%%%%%%%--Normal--semigroups--polynomial--decaying--rates--$\times$--spectral--properties--%%%%%%%%%%%%%%%%%%%%%%%%%%%%%%%%%%%%%%%%%%%%%%%

\subsection{Normal semigroups: polynomial decaying rates $\times$ spectral properties}

\noindent It follows from the Spectral Theorem that every normal operator $N$ on a Hilbert space~$\mathcal{H}$, with ${\mathbb{C}}_+ \subset \varrho(N)$, generates a normal $C_0$-semigroup of contractions; namely,
\[e^{tN} = \int_{\sigma(N)} e^{t\lambda} \quad dE^N(\lambda),\]
where $E^N$ is the resolution of the identity of $N$. It is known that every normal $C_0$-semigroup of contractions is of this form \cite{Rudin}.

We recall that every normal operator $N$ can be written as $N = N_R + iN_I$, where 
\[N_R = \frac{N+N^*}{2} \quad {\rm and} \quad N_I = -i\frac{N-N^*}{2}\]
are self-adjoint operators and $N_RN_I = N_IN_R$. In this case, $E^N$ corresponds  to a joint resolution of the identity associated with the operator pair $\{N_R,N_I\}$ \cite{Berezannskii,Samoilenko}. Thus, for every $x \in {\mathcal{H}}$, $\Vert x\Vert_{\mathcal{H}} = 1$, 

\begin{eqnarray}\label{real}
\nonumber \Vert e^{tN}x\Vert_{{\mathcal{H}}}^2 &=&\|e^{t(N_R+i N_I)}x\|_{{\mathcal{H}}}^2\\ \nonumber &=& \int_{\sigma(N_R) \times  \sigma(N_I)}  
|e^{t(y+iv)}|^2 \quad d\mu_x^{N_R}(y) \, d\mu_x^{N_I}(v)\\ \nonumber &=& \int_{\sigma(N_I)} 1 \,  d\mu_x^{N_I}(v) \, \int_{\sigma(N_R)} e^{2ty}  \quad d\mu_x^{N_R}(y)\\  &=& \int_{-\infty}^0   e^{2ty} d\mu_x^{N_R}(y)\,,
\end{eqnarray}
where $\mu_x^{N_R}$ denotes the spectral measure of~$N_R$ associated with $x$; the last equality in (\ref{real}) is a consequence from fact that $(e^{tN})_{t \geq 0}$ is a semigroup of contractions. Thus, at least when $\mu_x^{N_R}$ has a certain local regularity (with respect to the Lebesgue measure), we expect that   
\[\Vert e^{tN}x\Vert_{{\mathcal{H}}}^2 = \int_{-\infty}^0  e^{2ty} d\mu_x^{N_R}(y) ~\sim~ \mu_x^{N_R}(B(0,{1/t}))\,.\] 

If $f$ and $g$ are two real-value functions, $f\sim g$ means that $f$ and $g$ are asymptotically equivalent, that is, 
\[\lim_{t \to \infty}\frac{f(t)}{g(t)} =1\,.\]

In order to present our next result, we recall the following definition.

\begin{definition}{\rm Let $\mu$ be  a finite (positive) Borel measure on $\mathbb{R}$. The pointwise lower and upper scaling exponents of $\mu$  at $w \in \mathbb{R}$ are defined, respectively, by  
\[d_\mu^-(w) := \liminf_{\epsilon \downarrow 0} \frac{\ln \mu (B(w,\epsilon))}{\ln \epsilon} \quad{\rm  and }\quad d_\mu^+(w) := \limsup_{\epsilon \downarrow 0} \frac{\ln \mu (B(w,\epsilon))}{\ln \epsilon}\,,\]
if, for all small enough $\epsilon>0$,  $\mu(B(w,\epsilon))> 0$; $d_\mu^{\mp}(w) := \infty$\,, otherwise.}
\end{definition}

Taking into account (\ref{real}), the following result is expected. 

\begin{proposition}\label{decpolproposition} Let~$N$  be a normal  operator so that ${\mathbb{C}}_+ \subset \varrho(N)$, and let $x \in \mathcal{H}$, with $x \not = 0$. Then,
\begin{equation*}
d_{\mu_x^{N_R}}^+(0) = -\liminf_{t \to \infty} \frac{\ln  \|e^{tN}x\|_{{\mathcal{H}}}^2}{\ln t}  \quad\textrm{and}\quad   d_{\mu_x^{N_R}}^-(0) = -  \limsup_{t \to \infty} \frac{\ln \|e^{tN}x\|_{{\mathcal{H}}}^2}{\ln  t} \,.
\end{equation*}
\end{proposition}

We note that Proposition \ref{decpolproposition} relates, for every $x\in\mathcal{H}$, the polynomial decaying rates of $\|e^{tN}x\|_{{\mathcal{H}}}$ to dimensional properties of  the spectral measure $\mu_x^{N_R}$. Namely, this result establishes an explicit relation between the dynamics of the semigroup and the local scale spectral properties of its generator. 

We also note that Proposition \ref{decpolproposition} indicates that the polynomial decaying rates of an orbit $(e^{tN}x)_{t \geq 0}$ may depend on sequences of time going to infinity; by Proposition \ref{decpolproposition}, this will occur if $d_{\mu_x^{N_R}}^-(0) < d_{\mu_x^{N_R}}^+(0)$. We will show (Corollary \ref{cortheoremnormal}) that if $(e^{tN})_{t \geq 0}$ is stable but not exponentially stable, then, Baire generically in~$\mathcal{H}$,  
\[d_{\mu_x^{N_R}}^-(0) = 0 {\rm ~~~~and~~~~ } d_{\mu_x^{N_R}}^+(0) = \infty.\]
\begin{proof} [{Proof} {\rm (Proposition~\ref{decpolproposition})}] Let $\mu$ be a finite (positive) Borel measure on $\mathbb{R}$. We show that, for each $w \in {\mathbb{R}}$,
\begin{equation}\label{exploc1}
\liminf_{t \to \infty} \frac{\ln [ \int_{\mathbb{R}} e^{-2t|w-y|} d\mu(y)]}{\ln t} = -d_\mu^+(w),
\end{equation}
\begin{equation}\label{exploc2}
\limsup_{t \to \infty} \frac{\ln [ \int_{\mathbb{R}} e^{-2t|w-y|} d\mu(y)]}{\ln  t} = -d_\mu^-(w);
\end{equation}
so, Proposition \ref{decpolproposition}  becomes a direct consequence of~(\ref{real}).

Fix $w\in\mathbb{R}$. If there exists $\epsilon > 0$ such that $\mu(B(w,\epsilon)) = 0$, then, for each $t \geq 0$, 
\[ \int_{\mathbb{R}} e^{-2t|w-y|} d\mu(y) = \int_{B(w,\epsilon)^c} e^{-2t|w-y|} d\mu(y) \leq  \mu(\mathbb{R}) e^{-2t\epsilon};\]
 thus, 
\[\liminf_{t \to \infty} \frac{\ln[ \int_{\mathbb{R}} e^{-2t|w-y|} d\mu(y)]}{\ln t} = \limsup_{t \to \infty} \frac{\ln [\int_{\mathbb{R}} e^{-2t|w-y|} d\mu(y)]}{\ln t} = - \infty\]
 and the result follows. Suppose that, for each $\epsilon > 0$, $\mu(B(w,\epsilon)) > 0$. Since
\[ \int_{\mathbb{R}} e^{-2t|w-y|} d\mu(y)  \geq \int_{B(w,{1/t})} e^{-2t|w-y|} d\mu(y)  \geq e^{-2} \mu(B(w,{1/t})),\] 
it follows that
\[\liminf_{t \to \infty} \frac{\ln [ \int_{\mathbb{R}} e^{-2t|w-y|} d\mu(y)]}{\ln t} \geq - \limsup_{t \to \infty} \frac{\ln \mu(B(w,{1/t}))}{\ln t^{-1}} = -d_{\mu}^+(w)\]
 and that
\[\limsup_{t\to \infty} \frac{\ln [ \int_{\mathbb{R}} e^{-2t|w-y|} d\mu(y)]}{\ln t} \geq - \liminf_{t \to \infty} \frac{\ln \mu(B(w,{1/t}))}{\ln t^{-1}} = -d_{\mu}^-(w).\]

Now, let $0 < \delta < 1$. Then, for each $t>0$,
\begin{eqnarray}\label{exploc3}
\nonumber \!\!\! \int_{\mathbb{R}} e^{-2t|w-y|} d\mu(y) &=& \int_{B(w,\frac{1}{t^{1-\delta}})}  e^{-2t|w-y|} d\mu(y) +   \int_{B(w,\frac{1}{t^{1-\delta}})^c}  e^{-2t|w-y|} d\mu(y)\\
&\leq&  \mu\big(B\big(w,{1/t^{1-\delta}})) +   e^{-t^\delta} \mu(\mathbb{R}).
\end{eqnarray} 

Given that it is not possible to compare directly the  two terms on the right-hand side of~\eqref{exploc3}, one needs to analyze two distinct cases. 

\textbf{Case} $d_{\mu}^-(w) < \infty$: 
One has, from the definition of $d_{\mu}^-(w)$, that 
\[\liminf_{t \to \infty} \frac{\ln \mu(B(w,{1/t^{1-\delta}})) }{ \ln t^{-(1-\delta)}} = d_{\mu}^-(w) < \max \{2d_{\mu}^-(w),1\}=:\gamma \] 
($\gamma$ can be defined as any positive number greater than $d_{\mu}^-(w)$), so
\[\limsup_{t \to \infty} \frac{\ln \mu(B(w,{1/t^{1-\delta}}))) }{ \ln t^{1-\delta}} > - \gamma. \]
Hence, there exists a sequence $(t_k)$, with $\displaystyle\lim_{k \to \infty} t_k = \infty$, such that, for  sufficiently large~$k$,
\begin{equation}\label{exploc4}
\mu(B(w,{1/t_k^{1-\delta}})) \geq  t_k^{-\gamma(1-\delta)}  \geq e^{-t_k^\delta} \mu(\mathbb{R}).
\end{equation}
Now, combining (\ref{exploc3}) and (\ref{exploc4}) one has, for sufficiently large $k$,
\[\int_{\mathbb{R}} e^{-2t_k|w-y|} d\mu(y) \leq 2\, \mu\big(B\big(w,{1/t_k^{1-\delta}}\big)\big),\]
which results in 
\[\frac{1}{(1-\delta)} \liminf_{t \to \infty} \frac{\ln [ \int_{\mathbb{R}} e^{-2t|w-y|} d\mu(y)]}{\ln t} \leq - \limsup_{t\to \infty} \frac{\ln \mu(B(w,{1/t^{1-\delta}})))}{\ln t^{-(1-\delta)}} = -d_{\mu}^+(w).\]
Since $0 < \delta < 1$ is arbitrary, the complementary inequality in~\eqref{exploc1} follows.

It remains to prove the complementary inequality in~\eqref{exploc2}. This is trivial if $d_{\mu}^-(w)=0$, since $\displaystyle\limsup_{t \to \infty}$$ \ \ln [\int_{\mathbb{R}} e^{-2t|w-y|} d\mu(y)]/\ln t \le0$. So, let $d_{\mu}^-(w)>0$; it follows from the definition of $d_{\mu}^-(w)$ that, for each $0<\epsilon<d_{\mu}^-(w)$, there exists $t_{\delta, \epsilon}>0$ such that, for  $t > t_{\delta, \epsilon}$, 
\begin{equation}\label{exploc5}
\mu\big(B(w,{1/t^{1-\delta}}))\big) \leq  t^{-(1-\delta) (d_{\mu}^-(w)-\epsilon)}.
\end{equation}
 Combining (\ref{exploc3}) with (\ref{exploc5}), one gets, for sufficiently large $t$,
\[\int_{\mathbb{R}} e^{-2t|w-y|} d\mu(y) \leq 2 t^{-(1-\delta) (d_{\mu}^-(w)-\epsilon)}. \] Thus, 
\begin{equation*}
\limsup_{t\to \infty} \frac{\ln[\int_{\mathbb{R}} e^{-2t|w-y|} d\mu(y)]}{\ln t} \leq  -(1-\delta) (d_{\mu}^-(w)-\epsilon), 
\end{equation*}
and since $0<\delta <1$ and $0 < \epsilon < d_{\mu}^-(w)$ are arbitrary, the result follows.

\textbf{Case} $d_{\mu}^-(w) = \infty$:
One has
\[\lim_{t \to \infty} \frac{\ln \mu(B(w,{1/t}))}{\ln t}  = -\infty\,;\]
given an arbitrary $\alpha>0$, there is $t_\alpha>0$ so that, for each $t > t_\alpha$, $\mu(B(w,{1/t})) \leq t^{-2\alpha}$. Combining this inequality with~\eqref{exploc3} (taking $\delta = \frac{1}{2}$), one obtains, for sufficiently large $t$, 
\[ \int_{\mathbb{R}} e^{-2t|w-y|} d\mu(y) \leq  \mu\big(B\big(w,\frac{1}{t^{1/2}}\big)\big) +   e^{-t^{1/2}} \mu(\mathbb{R})   \leq 2 t^{-\alpha},\]
 from which follows that 
\[ \liminf_{t \to \infty} \frac{\ln [\int_{\mathbb{R}} e^{-2t|w-y|} d\mu(y)]}{\ln t}\le -\alpha\]
and that
\[ \limsup_{t \to \infty} \frac{\ln[\int_{\mathbb{R}} e^{-2t|w-y|} d\mu(y)]]}{\ln t}\le -\alpha;\] 
since $\alpha>0$ is arbitrary, the result follows. 
\end{proof}

%%%%%%%%%%%%%%%%%%%%%%%%%%%%%%%%%%%%%%%%%%%%%%%%%%%%%%%%%%%%%%%%%%%%%%--Proof--of--Theorem 1.2--%%%%%%%%%%%%%%%%%%%%%%%%%%%%%%%%%%%%%%%%%%%%%%%%%%%%%%%%%%%%%%%%%%%%%%%%%%%%%%%

\subsection{Proof of Theorem \ref{maintheoremnormal}}

\noindent Next, we present a proof of Theorem \ref{maintheoremnormal}. The main ingredient of this proof is the relation, given by the Spectral Theorem, between the decaying rates of the semigroup $(e^{tN})_{t\geq 0}$ and the local scale properties of the corresponding spectral measures of $N_R$. However, some preparation is required. 

We recall that a $C_0$-semigroup $(T(t))_{t \geq 0}$ is weakly stable if it converges to zero as $t \to \infty$ in the weak operator topology.

\begin{theorem}[Theorem 5.3 in~\cite{Muller}] \label{theoMuller1}Let $(T(t))_{t\geq 0}$ be a weakly stable $C_0$-semigroup on a Hilbert space~$\mathcal{H}$, with generator~$A$, such that $\omega_0(T) = 0$. Let $g: {\mathbb{R}}_+ \longrightarrow (0,\infty)$ be a bounded function such that $\displaystyle\lim_{t \to \infty} g(t) = 0$ and let $\epsilon>0$. Then, there exists $x_0 \in \mathcal{H}$  so that $\Vert x_0 \Vert_{\mathcal{H}} < \displaystyle\sup_{t \geq 0} \{g(t)\} + \epsilon$ and
\[\vert \langle T(t)x_0,x_0 \rangle \vert > g(t), \quad \forall t \geq 0.\]
\end{theorem}

\begin{lemma}\label{lemmanormal}Let $A$ be a negative self-adjoint operator so that  $0 \in \sigma(A)$, and let  $\alpha:{\mathbb{R}}_+ \longrightarrow (0,\infty)$ such that
\[\lim_{t \to \infty} \alpha(t) = \infty.\]
Then, there exist $x \in \mathcal{H}$ and a sequence $t_j \rightarrow \infty$ such that, for sufficiently large~$j$,  
\[\mu_{x}^A\big(B(0,{1/t_j})\big) \geq \frac{1}{\alpha(t_j)}.\]
\end{lemma}

\begin{proof} Since $0 \in \sigma(A)$, it follows from Gearhart-Pr\"{u}ss Theorem that $\omega_0(A) = 0$. We also note that it is sufficient to prove the case in which there exists a sequence $s_j \rightarrow \infty$ such that $\alpha(\sqrt{s_j}) \leq e^{\sqrt{s_j}}$, for sufficiently large~$j$. Set $g:{\mathbb{R}}_+  \longrightarrow (0,\infty)$ such that it satisfies the hypotheses of Theorem \ref{theoMuller1}, $\displaystyle\sup_{t\geq 0}\{g(t)\}<1$ and 
\begin{eqnarray}\label{lemm1}
\lim_{t \to \infty} \frac{1}{g(t)\alpha(\sqrt{t})} = 0.
\end{eqnarray}
Since $\sqrt{g}$ also satisfies the hypotheses of Theorem \ref{theoMuller1}, there exists $x \in {\mathcal{H}}$, $\Vert x \Vert_{\mathcal{H}} \leq 1$, such that for every $t >0$,
\begin{eqnarray}\label{lemm2}
\nonumber g(t) \leq \Vert e^{tA}x \Vert_{\mathcal{H}}^2 &=& \int_{\mathbb{R}} e^{2ty} d\mu_x^A(y)\\ \nonumber &=& \int_{B(0,\frac{1}{\sqrt{t}})}  e^{2ty} d\mu_x^A(y) +   \int_{B(0,\frac{1}{\sqrt{t}})^c}  e^{2ty} d\mu_x^A(y)\\
&\leq&  \mu_x^A\big(B\big(0,\frac{1}{\sqrt{t}}\big)\big) +   e^{-\sqrt{t}}.
\end{eqnarray} 

Now, if there does not exist a sequence $t_j \rightarrow \infty$ so that, for large enough~$j$,  
\[\mu_x^A(B(0,1/t_j)) \geq \frac{1}{\alpha(t_j)},\]
then, by (\ref{lemm2}), for large enough~$t$,
\[g(t) \leq \frac{1}{\alpha(\sqrt{t})} +   e^{-\sqrt{t}},\]
which implies, for large enough~$j$, 
\[\frac{2}{g(s_j) \alpha(\sqrt{s_j})} \geq 1;\]
since this contradicts (\ref{lemm1}), the result is proven. 
\end{proof}

\begin{proof}[{Proof} $(${\rm Theorem~\ref{maintheoremnormal}}$)$] Since $\sup \{R (\lambda): \lambda \in \sigma(N)\} = 0$ and $\sigma(N_R) \subset {\mathbb{R}}_-$ is closed, one has $0 \in \sigma(N_R)$. Therefore, by (\ref{real}), one can assume without loss of generality that $N$ is a self-adjoint operator such that $0 \in \sigma(N) \subset {\mathbb{R}}_-$. 

Since, for each $t \geq 0$, the mapping 
\[ {\mathcal{H}} \ni x\; \longmapsto \;\alpha(t)\, \Vert e^{tN}x \Vert_{{\mathcal{H}}} \]
is continuous, one has that %it follows that, for each $k \geq 1$ and each $n \geq 1$, the set 
%\[\bigcup_{t\geq k} \{x \mid \alpha(t)  \Vert e^{tN}x\Vert_{\mathcal{H}} > n\}\]
%is open, from which it follows that
\begin{eqnarray*}
{\mathcal{G}}_N(\alpha)&:=& \{x  \mid  \limsup_{t \to \infty} \alpha(t)  \Vert e^{tN}x\Vert_{\mathcal{H}} = \infty\}\\ &=& \bigcap_{n\geq 1} \bigcap_{k \geq 1} \bigcup_{t\geq k} \{x \mid  \alpha(t)  \Vert e^{tN}x\Vert_{\mathcal{H}} > n \}
\end{eqnarray*}
is a $G_\delta$ set in ${\mathcal{H}}$. The proof that 
\[{\mathcal{G}}_N(\beta):= \{x  \mid  \liminf_{t \to \infty} \beta(t)  \Vert e^{tN}x\Vert_{\mathcal{H}} = 0\}\]
is a $G_\delta$ set in ${\mathcal{H}}$ is completely analogous, being, therefore, omitted. 

Let $\overline{x} \in \mathcal{H}$, let $(t_j)$ be the sequence given by Lemma \ref{lemmanormal} and set, for every $x \in \mathcal{H}$ and every $k \geq 1$, 
\[x_k := E^{N}(S_k) x + \frac{1}{k}\overline{x},\]
where $S_k := (-\infty,-1/k) \cup \{0\} \cup (1/k,\infty)$. It is clear that $x_k \rightarrow x$ in ${\mathcal{H}}$.

Now, since $(e^{tN})_{t \geq 0}$ is stable, $E^{N}(\{0\}) = 0$. Namely, it follows from the Spectral Theorem and dominated convergence that, for each $x\in\mathcal{H}$, 
\begin{equation*}
\lim_{t\to \infty}\Vert e^{tN}x\Vert^2=\mu_x^N(\{0\})+\lim_{t\to \infty}\int_{{\mathbb{R}}_-\setminus\{0\}}e^{2ty}d\mu_x^N(y)=\mu_x^N(\{0\});    
\end{equation*}
therefore, $(e^{tN})_{t \geq 0}$  is stable if, and only if, $0$ is not an eigenvalue of $N$, that is, if and only if $E^{N}(\{0\}) = 0$. Hence, for each $k\geq 1$ and each~$j$ such that $\frac{1}{t_j}<\frac{1}{k}$, one has
\begin{eqnarray*}
\mu_{x_k}^N (B(0,\frac{1}{t_j})) &=& \langle E^N(B(0,{1/t_j})) E^{N}(S_k)x, E^{N}(S_k)x \rangle  + \frac{1}{k^2} \langle E^N(B(0,{1/t_j}))\overline{x}, \overline{x} \rangle\\ &=& \langle E^{N}(\{0\})x, E^{N}(S_k)x \rangle  + \frac{1}{k^2} \langle E^N(B(0,{1/t_j}))\overline{x}, \overline{x} \rangle\\ &=& \frac{1}{k^2} \mu_{\overline{x}}^N (B(0,{1/t_j})),
\end{eqnarray*}
from which follows that, for sufficiently large~$j$, 
\begin{eqnarray*}
\alpha(t_j)\Vert e^{t_jN}x_k \Vert_{\mathcal{H}} &\geq& \alpha(t_j) \biggr(\int_{B(0,{1/t_j})}  e^{2t_jy} d\mu_{x_k}^N(y) \biggr)^{1/2}\\ 
&\geq& \frac{\alpha(t_j)}{e}  \biggr(\mu_{x_k}^N (B(0,{1/t_j})) \biggr)^{1/2}\\ &=&  \frac{\alpha(t_j) }{ke} \biggr(\mu_{\overline{x}}^N (B(0,{1/t_j})) \biggr)^{1/2}\\ &\geq& \frac{\sqrt{\alpha(t_j)}}{ke}. 
\end{eqnarray*}
Consequently, for every $k \geq 1$, 
\[\limsup_{t \to \infty} \alpha(t)\Vert e^{tN}x_k \Vert_{\mathcal{H}} = \infty;\]
this proves that ${\mathcal{G}}_N(\alpha)$ is a dense set in~$\mathcal{H}$.

Now we prove that
\[{\mathcal{G}}_N(\beta):= \{x \in  {\mathcal{H}} \mid  \liminf_{t \to \infty} \beta(t)  \Vert e^{tN}x\Vert_{\mathcal{H}} = 0\}\]
is a dense set in~$\mathcal{H}$. Given $x \in \mathcal{H}$, define, for each $k \geq 1$, 
\[x_k := E^{N}(S_k)x.\]
Then, $x_k \rightarrow x$ in ${\mathcal{H}}$. Moreover, since $E^{N}(\{0\}) = 0$, it follows that for each $\epsilon< \frac{1}{k}$, $\mu_{x_k}^N (B(0,\epsilon)) =0$. Thus,
\begin{equation}\label{spt2}
\lim_{t \to \infty}\beta(t)\Vert e^{tN}x_k \Vert_{\mathcal{H}}\; \leq \;\lim_{t \to \infty} \beta(t) e^{-t\epsilon} \Vert x_k \Vert_{\mathcal{H}}= 0. 
\end{equation} 

It remains to prove that the assumption on $\beta$ is optimal. Let $x \in \mathcal{H}$ be such that there exists $\epsilon > 0$ with $\mu_x^N([-\epsilon,0])=0$; then, $\Vert e^{tN}x\Vert_{\mathcal{H}} = O(e^{-t\epsilon})$. Set  $\alpha(t) =t$ and consider ${\mathcal{G}}_N(\alpha)$,  which is a dense $G_\delta$ in~$\mathcal{H}$. Then, for each $x \in {\mathcal{G}}_N(\alpha)$  and each $\epsilon>0$, $\mu_x^N([-\epsilon,0])>0$. It follows from the Spectral Theorem and Jensen's inequality that 
\begin{eqnarray*}
\frac{1}{\Vert e^{tN}x\Vert_{\mathcal{H}}^2} &=& \biggr(\int_{-\infty}^0 e^{2ty} d\mu_x^N(y)\biggr)^{-1}\leq \biggr(\int_{-\epsilon}^0 e^{2ty} d\mu_x^N(y)\biggr)^{-1} \\ &\leq& \frac{1}{(\mu_x^N([-\epsilon,0]))^2} \int_{-\epsilon}^0 e^{-2ty} d\mu_x^N(y) \leq \frac{ e^{2t\epsilon}}{\mu_x^N([-\epsilon,0])}.
\end{eqnarray*}
This shows that, for each $x \in {\mathcal{G}}_N(\alpha)$, $\Vert e^{tN}x\Vert_{\mathcal{H}}$ vanishes slower than exponential as $t \rightarrow \infty$.
\end{proof}

The next result, which is a direct consequence of Proposition \ref{decpolproposition} and Theorem~\ref{maintheoremnormal}, indicates how delicate is the relation between the dynamics of the semigroup and the spectral properties of its generator.
 
\begin{corollary}\label{cortheoremnormal} 
Let $N$ be as in the statement of Theorem~\ref{maintheoremnormal}. Suppose that $(e^{tN})_{t\geq 0}$ is stable. Then, 
\[G_N := \{x \mid d_{\mu_x^{N_R}}^-(0) = 0  \quad and \quad d_{\mu_x^{N_R}}^+(0) = \infty\}\]
is a dense $G_\delta$ set in~$\mathcal{H}$.
\end{corollary} 

\begin{proof} For every $n \geq 1$ and every  $t>0$, consider $\alpha_n(t):=t^{1/n}$ and $\beta_n(t):=t^{n}$. Then, by Proposition~\ref{decpolproposition},
\[G_N = \bigcap_{n \geq 1} {\mathcal{G}}_N(\alpha_n,\beta_n).\]
It follows from Theorem \ref{maintheoremnormal} that $G_N$ is a dense $G_\delta$ set in~$\mathcal{H}$.
\end{proof}

\begin{remark}\label{remarktheonorm} In proof of Theorem \ref{maintheoremrefineddecay} we exhibit another proof that  ${\mathcal{G}}_N(\alpha)$ is a dense set in ${\mathcal{H}}$. The above proof  explicitly indicates the relation between the decaying rates of the semigroup $(e^{tN})_{t\geq 0}$ and the local scale properties of the corresponding spectral measures of $N_R$ (Lemma \ref{lemmanormal}).
\end{remark}

%%%%%%%%%%%%%%%%%%%%%%%%%%%%%%%%%%%%%%%%%%%%%%%%%%%%%%%%%%%%%%%%%%%%%%--Proof--of--Theorem 1.3--%%%%%%%%%%%%%%%%%%%%%%%%%%%%%%%%%%%%%%%%%%%%%%%%%%%%%%%%%%%%%%%%%%%%%%%%%%%%%%%

\subsection{Proof of Theorem~\ref{theoremnormalclas}}

\begin{proof}[{Proof} {\rm (Theorem~\ref{theoremnormalclas})}] Again, by (\ref{real}), we assume without loss of generality that $N$ is a self-adjoint operator such that $N \leq 0$.

\textbf{{$i$})}. This is a direct consequence of Gearhart-Pr\"{u}ss Theorem.

\textbf{$ii$)}. This is a consequence of Theorem~\ref{maintheoremnormal}, Gearhart-Pr\"{u}ss Theorem and  the fact that $(e^{tN})_{t \geq 0}$  is stable if, and only if, $0$ is not an eigenvalue of $N$.

Suppose that for every $\alpha$ and $\beta$ defined as in the statement of Theorem~\ref{maintheoremnormal}, ${\mathcal{G}}_N(\alpha,\beta)$ is a dense $G_\delta$ set in~$\mathcal{H}$. Let $\alpha(t) = \beta(t) := t$, for all $t >0$. If $x \in \mathcal{H} $ and $(x_n)_n \subset {\mathcal{G}}_N(\alpha,\beta)$ are such that $\displaystyle\lim_{n \to \infty} x_n = x$, then, by Moore-Osgood Theorem, 
\[ \lim_{t \to \infty} \Vert e^{tN}x \Vert_{\mathcal{H}} =  \lim_{n \to \infty} \lim_{t \to \infty}  \Vert e^{tN}x_n \Vert_{\mathcal{H}} =  \lim_{n \to \infty}  \liminf_{t \to \infty} \Vert e^{tN}x_n \Vert_{\mathcal{H}} = 0.\]
This shows that $(e^{tN})_{t \geq 0}$  is stable and, therefore, that $0$ is not an eigenvalue of $N$. We note that if $x \in {\mathcal{G}}_N(\alpha,\beta)$, then 
\[\limsup_{t \to \infty} t \Vert e^{tN}x \Vert_{\mathcal{H}} = \infty.\]
Hence, $(e^{tN})_{t \geq 0}$ is not exponentially stable and, since $\sigma(N) \subset {\mathbb{R}}_-$, it follows from Gearhart-Pr\"{u}ss Theorem that $0 \in \sigma(N)$. 
  
The reciprocal is a direct consequence of Theorem~\ref{maintheoremnormal}.
  
\textbf{$iii$)}. If there exists an $x \in \mathcal{H}$ such that 
\begin{equation*}
\lim_{t\to \infty}\Vert e^{tN}x \Vert_{\mathcal{H}}^2=\mu_x^N(\{0\})>0,    
\end{equation*}
then it follows that $0$ is an eigenvalue of $N$. Therefore, it remains to prove that
\[\mathcal{F}_N:=\{x \in \mathcal{H}\mid \lim_{t\to\infty}\Vert e^{tN}x\Vert_{\mathcal{H}}>0\}\] %such that $(e^{tN}x)_{t\geq 0}$ does not converge to zero
is a dense open set in~$\mathcal{H}$ if $0$ is an eigenvalue of $N$. Since the mapping 
\[ {\mathcal{H}} \ni x\; \longmapsto \; \Vert E^{N}(\{0\})x \Vert_{\mathcal{H}}^2 \]
is continuous, it follows that 
\begin{eqnarray}
  \nonumber \mathcal{F}_N = \{x \mid  \mu_x^N(\{0\}) > 0 \}=\{x \mid  \Vert  E^{N}(\{0\})x \Vert_{\mathcal{H}}^2 > 0 \}
\end{eqnarray}
is a open set in ${\mathcal{H}}$ (we note also that $\mathcal{F}_N=\mathcal{H}\setminus{\rm kernel}(N)^\perp$).

Now, given $x \in \mathcal{H}$, write $x = x_1 + x_2$, with $x_1 \in {\mathrm {Span}}\{x_0\}^\perp$ and $x_2 \in {\mathrm {Span}}\{x_0\}$, where $x_0$, with $\Vert x_0 \Vert_{\mathcal{H}} =1$, is an eigenvector of $N$ associated with the eigenvalue $0$. If $x_2 \not =  0$, then 
\begin{eqnarray*}
\nonumber \Vert e^{tN} x \Vert_{\mathcal{H}}^2 &=&   \mu_{x}^N(\{0\})+\int_{{\mathbb{R}}_-\setminus\{0\}}e^{2ty}d\mu_{x}^N(y)  \geq  \mu_{x}^N (\{0\}) = \Vert E^N(\{0\})x \Vert_{\mathcal{H}}^2\\  \nonumber &\geq& 2\,R\big\langle E^N(\{0\})x_1,E^N(\{0\})x_2 \big\rangle + \Vert E^N(\{0\})x_2 \Vert_{\mathcal{H}}^2\\ &=&  \Vert x_2 \Vert_{\mathcal{H}}^2, 
\end{eqnarray*}
from which follows that 
\[\lim_{t \to \infty}  \Vert e^{tN} x \Vert_{\mathcal{H}} > 0.\]
Now, if $x_2 =  0$, define, for each $k\geq1$,
\[x_k  := x + \frac{x_0}{k}.\]
It is clear that $x_k \rightarrow x$. Moreover, by the previous arguments, one has, for each $k\geq1$,
\[\lim_{t \to \infty} \Vert e^{tN} x_k \Vert_{\mathcal{H}} > 0.\]
This proves that $\mathcal{F}_N$ is a dense open set in~$\mathcal{H}$.
\end{proof}

%%%%%%%%%%%%%%%%%%%%%%%%%%%%%%%%%%%%%%%%%%%%%%%%%%%%%%%%%%%%%%%%%%%%%%--Proof--of--Theorem 1.4--%%%%%%%%%%%%%%%%%%%%%%%%%%%%%%%%%%%%%%%%%%%%%%%%%%%%%%%%%%%%%%%%%%%%%%%%%%%%%%%

\subsection{Proof of Theorem \ref{maintheoremrefineddecay}}

\noindent The proof of Theorem \ref{maintheoremrefineddecay} relies, again, on Theorem \ref{theoMuller1}.

\begin{proof}[{Proof} {\rm (Theorem \ref{maintheoremrefineddecay})}] The proof that each one of the sets 
\[{\mathcal{G}}_A(\alpha):= \{x  \mid  \limsup_{t \to \infty} \alpha(t)  \Vert T(t)x\Vert_{\mathcal{H}} = \infty\}\]
and
\[{\mathcal{G}}_A(\beta):= \{x  \mid  \liminf_{t \to \infty} \beta(t)  \Vert T(t)x\Vert_{\mathcal{H}} = 0\}\]
is a $G_\delta$ set in $\mathcal{H}$ follows the same reasoning presented in the proof of Theorem~\ref{maintheoremnormal}.

Since, by hypothesis~(\ref{H2}), there exists $C>0$ such that, for every $x \in {\mathcal{D}}(A^k)$ and every sufficiently large~$t$, 
\[\Vert T(t)x\Vert_{\mathcal{H}} \leq C r(t) \Vert A^kx \Vert_{\mathcal{H}},\]
it follows that, for every $x \in {\mathcal{D}}(A^k)$,  
\[\liminf_{t \to \infty} \beta(t)  \Vert T(t)x\Vert_{\mathcal{H}} = 0.\]
Thus, ${\mathcal{G}}_A(\beta) \supset {\mathcal{D}}(A^k)$ is a dense set in ${\mathcal{H}}$.

It remains to prove that ${\mathcal{G}}_A(\alpha)$ is dense in~$\mathcal{H}$. We note that, by Theorem~\ref{theoMuller1}, there exists $x_0 \in \mathcal{H}$ such that $\Vert x_0 \Vert_{\mathcal{H}} \leq 1$ and, for each $t\geq 0$, 
\begin{equation}\label{maintheorem1}
\Vert T(t)x_0 \Vert_{\mathcal{H}} > \frac{1}{\sqrt{\alpha(t)+2}}.  
\end{equation}

So, given $x \in \mathcal{H}$, suppose now that, for each $k\geq 1$, there exists a sequence $t_j \rightarrow \infty$ so that, for each $j$,
\begin{equation}\label{maintheorem2}
\Vert T(t_j)x \Vert_{\mathcal{H}} < \frac{\Vert T(t_j)x_0 \Vert_{\mathcal{H}}}{4k};    
\end{equation}
otherwise, it follows from (\ref{maintheorem1}) that $x \in {\mathcal{G}}_A(\alpha)$. Set, for each $k\geq 1$, 
\[x_k := x + \frac{x_0}{k}.\]
It is clear that $x_k \rightarrow x$ in~$\mathcal{H}$. Moreover, by (\ref{maintheorem2}), for each $k \geq 1$ and each $j$,
\begin{eqnarray}\label{maintheorem3}
\nonumber \Vert T(t_j)x_k\Vert_{\mathcal{H}}^2 &\geq& - \frac{2\Vert T(t_j)x \Vert_{\mathcal{H}} \Vert T(t_j)x_0 \Vert_{\mathcal{H}}}{k} + \frac{\Vert T(t_j)x_0 \Vert_{\mathcal{H}}^2}{k^2}\\ &\geq& \frac{\Vert T(t_j)x_0 \Vert_{\mathcal{H}}^2}{2k^2}.   
\end{eqnarray}
Thus, combining (\ref{maintheorem1}) and (\ref{maintheorem3}) it follows that,
\[\displaystyle\limsup_{t \to \infty} \alpha(t) \Vert T(t)x_k\Vert_{\mathcal{H}} = \infty.\]
 for each $k \geq 1$. %This shows that $x \in \overline{{\mathcal{G}}_A(\alpha)}$, concluding the proof of the theorem. 
\end{proof}

%%%%%%%%%%%%%%%%%%%%%%%%%%%%%%%%%%%%%%%%%%%%%%%%%%%%%%%%%%%%%%%%%%%%%%%%%%--Acknowledgments--%%%%%%%%%%%%%%%%%%%%%%%%%%%%%%%%%%%%%%%%%%%%%%%%%%%%%%%%%%%%%%%%%%%%%%%%%%%%%%%%%%%%%%%%%%%%%%%%%%%%%%%%%%%%%%%%%%%%%%%%%%%%%%%%%%%%%%%%%%%%%%%%%%%%%%%%%%%%%%%%%%%%%%%%%%%%%%%%%%%%%%%%%%%%%%%%%%%%%%%%%%%%%%%%%%%%%%%%%%%%%%%%%%%%%%%%%%%%%%%%%%%%%%%%%%%%%%%%

\begin{center} \Large{Acknowledgments} 
\end{center}

M.A.\ was supported by CAPES (a Brazilian government agency). S.L.C.\ was partially supported by FAPEMIG (a Brazilian government agency; Universal Project 001/17/CEX-APQ-00352-17). C.R.dO.\ thanks the partial support by CNPq (a Brazilian government agency, under contract 303503/2018-1). The authors are grateful to Pedro T.\  P.\ Lopes for fruitful discussions and helpful remarks.  We thank the anonymous referee for valuable suggestions that have substantially improved the exposition of the manuscript.

%%%%%%%%%%%%%%%%%%%%%%%%%%%%%%%%%%%%%%%%%%%%%%%%%%%%%%%%%%%%%%%%%%%%%%%%%%%%%%%%%%%%%%%%%%%%%%%%%%%%%%%%%%%%%%%%%%%%%%%%%%%%%%%%%%%%%%%%%%%%%%%%%%%%%%%%%%%%%%%%%%%%%%%%%%%%%%%%%%%%%%%%%%%%%%%%%%%%%%%%%%%%%%%%%%%%%%%%%%%%%%%%%%%%%%%%%%%%%%%%%%%%%%%%%%--thebibliography--%%%%%%%%%%%%%%%%%%%%%%%%%%%%%%%%%%%%%%%%%%%%%%%%%%%%%%%%%%%%%%%%%%%%%%%%%%%%%%%%

\noindent  Email: moacir@ufam.edu.br, Departamento de Matem\'atica, ~UFAM, Manaus, AM, 369067-005 Brazil

\noindent  Email: silas@mat.ufmg.br, Departamento de Matem\'atica, UFMG, Belo Horizonte, MG, 30161-970 Brazil

\noindent  Email: oliveira@dm.ufscar.br,  Departamento  de  Matem\'atica,   UFSCar, S\~a Carlos, SP, 13560-970 Brazil


\begin{thebibliography}{30}

\bibitem{Anantharaman} N. Anantharaman and M. L\'{e}autaud, {\it Sharp polynomial decay rates for the damped wave equation on the torus},  With an appendix by S. Nonnenmacher. Anal. PDE. {\bf 7} (2014), 159--214.

\bibitem{Arendt} W. Arendt and C. J. K. Batty, {\it Tauberian theorems and stability of one-parameter semigroups}, Trans. Amer. Math. Soc. {\bf 306} (1988), 837--852.

\bibitem{Batty} C. J. K. Batty, R. Chill, and Y. Tomilov, {\it Fine scales of decay of operator semigroups}, J. Eur. Math. Soc.  {\bf 18} (2016), 853--929.
 
\bibitem{Batty1} C. J. K. Batty and T. Duyckaerts, {\it Non-uniform stability for bounded semi-groups on Banach spaces}, J. Evol. Eq. {\bf 8} (2008), 765--780.

\bibitem{Berezannskii} Yu. M. Berezannskii, {\it The projection spectral theorem}, Russ. Math. Surveys. {\bf 39} (1984), 1--162. 

\bibitem{Borichev} A. Borichev and Y. Tomilov, {\it Optimal polynomial decay of functions and operator semigroups}, Math. Ann. {\bf 347} (2010), 455--478.

\bibitem{Burq} N. Burq, {\it D\'ecroissance de l\'energie locale de l\'equation des ondes pour le probl\'eme ext\'erieur et absence de r\'esonance au voisinage du r\'eel}, Acta Math. {\bf 180} (1998), 1--29.

\bibitem{Conti}  M. Conti, V. Danese, C. Giorgi, and V. Pata, {\it A model of viscoelasticity with time-dependent memory kernels}, Amer. J. Math. {\bf 140} (2018), 349--389.

\bibitem{Oliveira} C. R. de Oliveira, Intermediate spectral theory and quantum dynamics, Birkh\"auser, Basel, 2009. 

\bibitem{Herbst} I.W. Herbst, {\it The spectrum of Hilbert space semigroups}, J. Operator Th. {\bf 10} (1983), 87--94.

\bibitem{Howland} J. S. Howland, {\it On a theorem of Gearhart}, Integral Eq. Operator Th. {\bf 7} (1984), 138--142.

\bibitem{Korevaar} J. Korevaar, {\it On Newman\'s quick way to the prime number theorem}, Math. Intelligencer, {\bf 4} (1982), 108--115.

\bibitem{Lyubich} Yu. I. Lyubich and V\~u Q. Ph\'ong, {\it Asymptotic stability of linear differential equations on Banach spaces}, Studia Math. {\bf 88} (1988), 37--42.

\bibitem{Muller} V. M\"{u}ller and Y. Tomilov, {\it ``Large'' weak orbits of $C_0$-semigroups}, Acta Sci. Math. (Szeged). {\bf 79}  (2013), 475--505.

\bibitem{Pruss} J. Pr\"{u}ss, {\it On the spectrum of $C_0$-semigroups}, Trans. Amer. Math. Soc. {\bf 284} (1984), 847--857. 

\bibitem{Rozendaal} J. Rozendaal, D. Seifert, and R. Stahn, {\it Optimal rates of decay for operator semigroups on Hilbert spaces}, Advances in Mathematics {\bf 346} (2019), 359--388.

\bibitem{Rudin} W. Rudin, Functional analysis. Second edition. McGraw-Hill, New York, 1991.

\bibitem{Samoilenko} Y. S. Samoilenko, Spectral theory of families of self-Adjoint operators,  Kluwer, Dordrecht, 1991.

\bibitem{Stahn} R. Stahn, {\it Optimal decay rate for the wave equation on a square with constant damping on a strip}, Z. Angew. Math. Phys. {\bf 68}(2), 36 (2017).    

\bibitem{van}  J. M. A. M. van Neerven, The asymptotic behavior of semigroups of linear operators, Operator Theory: Advances and Applications, Birkh\"auser, Basel, 1996.

\end{thebibliography}
\end{document}